\documentclass[12pt]{amsart}
\usepackage{amsmath,amssymb}

\newtheorem{lemma}{Lemma}
\newtheorem{propozycja}{Proposition}
\newtheorem{proposition}{Proposition A\hskip-1pt}
\newtheorem{corollary}{Corollary}
\newtheorem{Theorem}{Theorem}
\newtheorem{uwaga}{Remark}
\newtheorem{definition}{Definition}

\newcommand{\C}{\Bbb C}

\def\Re{\operatorname{Re}}

\title{Gromov (non-)hyperbolicity of certain domains in $\mathbb{C}^{n}$}

\author{Nikolai Nikolov}
\author{Pascal J.~Thomas}
\author{Maria Trybu\l{}a}

\address{Institute of Mathematics and Informatics\\Bulgarian Academy
of Sciences\\ Acad. G. Bonchev 8, 1113 Sofia, Bulgaria\newline
\indent Faculty of Information Sciences\\
State University of Library Studies and Information Technologies\\
Shipchenski prohod 69A, 1574 Sofia,
Bulgaria}\email{nik@math.bas.bg}
\address{Universit\'e de Toulouse\\ UPS, INSA, UT1, UTM \\
Institut de Math\'e\-ma\-tiques de Toulouse\\
F-31062 Toulouse, France} \email{pascal.thomas@math.univ-toulouse.fr}
\address{Institute of Mathematics, Faculty of Mathematics and Computer Science
\\ Jagiellonian University, \L ojasiewicza 6, 30-348 Krak\'ow, Po\-land}\email{maria.trybula@im.uj.edu.pl}


\thanks{Research of the third author is supported by the International PhD
programme ``Geometry and Topology in Physical Models'' of the Foundation
for Polish Science, by the Polish National Science Center -- grant
PRO-2013/11/N/ST1/03609, and by the Bulgarian National Science Found -- contract DFNI-I 02/14.
The initial version of this paper was prepared during her visit to the Institute
of Mathematics and Informatics, Bulgarian Academy of Science, October 2013 -- April 2014.\\
\indent The authors would like to thank the referee for his/her valuable comments.}

\begin{document}

\begin{abstract}
We prove the Gromov non-hyperbolicity with respect to the Kobayashi distance for $\mathcal{C}^{1,1}$-smooth convex domains
in $\mathbb{C}^{2}$ which contain an analytic disc in the boundary or have a point of infinite type with rotation symmetry.
The same is shown for ``generic'' product spaces, as well as for the symmetrized polydisc and the tetrablock.
On the other hand, examples of smooth, non-pseudoconvex, Gromov hyperbolic domains in $\Bbb C^n$ are given.
\end{abstract}

\maketitle

\section{Introduction and statements}
In \cite{Gromov}, Gromov introduced the notion of almost hyperbolic space. He discovered that ``negatively curved'' space equipped with some distance share many properties with the prototype, even though the distance does not come from a Riemannian metric. This gave the impulse to intensive research to find new interesting classes of spaces which are hyperbolic in that sense. In this paper we are mainly interested in investigating this concept with respect to the Kobayashi distance of convex domains. One may suspect that it is a restriction to consider only the Kobayashi metric.
Actually, because the Kobayashi distance of a ($\Bbb C$-)convex domain containing no complex lines, as well as of a bounded strictly pseudoconvex domain, is bilipschitz equivalent to the (inner) Carath\'{e}odory and Bergman distances (see \cite[Theorem 12]{NPZ} and \cite[Proposition 4]{Nikolov}), it does not matter which one we choose (see below). Recall that a set $E$ in $\Bbb C^n$ is called \emph{$\Bbb C$-convex} if any intersection of $E$ with a complex line $l$ and its complement in $l$ are both connected in $l$ (cf. \cite{APS}).

The notion of a bilipschitz equivalence has the following generalization.

\begin{definition}
Let $(X_1,d_1)$ and $(X_2,d_2)$ be two metric spaces.
Then a map $\varphi:X_1\to X_2$ is said to be a \emph{quasi-isometry} if there are constants $c_1,c_2 >0$ such
that for any $x,y\in X_1$,
$$
c_1^{-1} d_1(x,y) - c_2 \le d_2\left( \varphi(x), \varphi(y) \right) \le c_1 d_1(x,y) + c_2.
$$
Two distances $d_1,d_2$ on a set $X$ are said to be quasi-isometrically equivalent
if the identity map is a quasi-isometry from $(X,d_1)$ to $(X,d_2)$.
\end{definition}

Gromov hyperbolicity is well-known to be invariant under bijective quasi-isometries of path metric spaces
(cf. \cite[Theorems 3.18, 3.20]{Jesus}).

\begin{definition}
\label{defghyp}
Let $(D,d)$ be a metric space. Given points $x,y,z\in D,$ the \emph{Gromov product} is
$$(x,y)_{z}=d(x,z)+d(z,y)-d(x,y).
$$
Let
$$
S_d(p,q,x,w)= \min\{(p,x)_{w},(x,q)_{w}\}-(p,q)_{w}.
$$
$(D,d)$ is \emph{Gromov hyperbolic} if
$$\sup_{p,q,x,w\in D} S_d(p,q,x,w) <\infty.
$$
If $S_d(p,q,x,w)\le 2\delta$, then $(D,d)$ is called \emph{$\delta$-hyperbolic}.
\end{definition}

We refer to \cite{Jesus} for other characterizations of Gromov hyperbolicity, especially
for path metric spaces. We chose this one because it does not use geodesics explicitly.

\begin{definition}
$(D,d)$ is a \emph{path metric space} if, for any two points
$x,y\in D$ and any number $\varepsilon >0,$ there exists a
rectifiable path joining $x$ and $y$ with length at most $d(x,y)+\varepsilon.$
Then the distance $d$ is called \emph{intrinsic}.
\end{definition}

From now on, let $D$ be a domain in $\mathbb{C}^{n}$.

Denote by $c_D$ and $l_{D}$ the Carath\'{e}odory distance and the Lempert function of $D$:
$$c_{D}(z,w)=\sup\{\tanh^{-1}|f(w)|:f\in\mathcal{O}(D,\Bbb D), f(z)=0\},$$
$$l_{D}(z,w)=\inf\{\tanh^{-1}|\alpha|:\exists\varphi\in\mathcal{O}(\mathbb{D},D)
\hbox{ with }\varphi(0)=z,\varphi(\alpha)=w\},$$
where $\mathbb{D}$ is the unit disc. The Kobayashi distance $k_{D}$ is the
largest pseudodistance not exceeding $l_{D}.$ The inner Carath\'{e}odory distance $c_D^i$ is
the inner pseudodistance associated to $c_D.$ So, $c_D\le c_D^i\le k_D\le l_D.$
By Lempert's seminal paper \cite{Lempert}, we have equalities above if $D$ is convex (or
bounded, $\mathcal{C}^2$-smooth and $\Bbb C$-convex).

An important  property of $k_{D}$ is that it is the integrated form of the Kobayashi metric $\kappa_{D}$
of $D,$ i.e. \begin{multline*}
k_{D}(z,w)=\inf\{\int_0^1\kappa_{D}(\gamma (t);\gamma^{\prime}(t))dt:\\
\gamma:[0,1]\to D\textup{ is a smooth curve with }\gamma(0)=z\textup{ and }\gamma(1)=w\},
\end{multline*}
where
$$\kappa_{D}(z;X)=\inf\{|\alpha|:\exists\varphi\in\mathcal{O}(\mathbb{D},D)
\hbox{ with }\varphi(0)=z,\ \alpha\varphi'(0)=X\},$$
$z,w\in D,\ X\in\mathbb{C}^{n}.$

We refer to \cite{Jarnicki} for basic properties of the invariants defined here and
of the Bergman distance $b_D.$

We shall say that $D$ is Gromov \emph{$s$-hyperbolic} if $(D,s_D)$ is Gromov hyperbolic
with respect to the distance $s$
(this should not be confused with $\delta$-hyperbolicity for some
constant $\delta>0$).

The first result concerning Gromov $k$-hyperbolicity for domains in $\C^n$
was given by Ba\-logh and Bonk \cite{Bonk} who gave both positive and negative examples.
They proved that any bounded strictly pseudoconvex domain is Gromov $k$-hyperbolic \cite[Theorem 1.4]{Bonk}.
They also showed that the Cartesian product of bounded strictly pseudoconvex domains is not Gromov $k$-hyperbolic
\cite[Proposition 5.6]{Bonk} which is a special case of a general situation mentioned in many places, but without proof (cf. \cite{Gaussier}).

\begin{propozycja}\label{I}
Assume that $(X_1,d_1)$ is a path metric space with $d_1$
unbounded and $(X_2,d_2)$ a metric space with unbounded $d_2$. Let
$d=\max\{d_1,d_2\}$. Then $(X_1\times X_2,d)$ is
not Gromov hyperbolic.
\end{propozycja}\label{I}

The next proposition is more general than the previous one. However its proof uses Proposition \ref{I}.

\begin{propozycja}\label{II} Let $(X_1,d_1)$ and $(X_2,d_2)$ be metric spaces,
such that one of them is a path metric space. Let $d=\max\{d_1,d_2\}$.
Then $(X_1\times X_2,d)$ is Gromov hyperbolic if and only if one of the factors is Gromov hyperbolic and the metric
of the second one is bounded (in particular, it is also Gromov hyperbolic).
\end{propozycja}

Moreover, the proof of Proposition \ref{I} and Remark 1 (following this proof)
show that the path property in
Proposition \ref{II} can be replaced the following.

\begin{definition}
A metric space $(Y,d)$ admits the \emph{weak midpoints property} if either $d$ is
bounded or there exist sequences $(x_k),(y_k),(z_k)\subset Y$ such that
$d(x_k,y_k)\to\infty$ and
\begin{equation}
\label{wmp}
\frac{d(x_k,z_k)}{d(x_k,y_k)}\to\frac{1}{2},\ \frac{d(y_k,z_k)}{d(x_k,y_k)}\to\frac{1}{2}.
\end{equation}
\end{definition}

\begin{corollary}\label{III} Let $D_1$ and $D_2$ be Kobayashi hyperbolic domains (i.e. $k_{D_1}$
and $k_{D_2}$ are distances) admitting non-constant bounded holomorphic functions
(for example, bounded domains). Then $D_1\times D_2$ is not Gromov $k$-hyperbolic.
\end{corollary}

To see this, it is enough to observe that if a domain $G$ in $\Bbb C^n$ admits a non-constant bounded
holomorphic function $f$ and $|f(z_j)|\to\sup_G|f|,$ then $k_G(z,z_j)\ge c_G(z,z_j)\to\infty.$
\smallskip

Note also that Proposition \ref{I} implies that if $D_1$ and $D_2$ are planar domains with complements containing more than one point
(i.e. they are Kobayashi hyperbolic), then
$D_1\times D_2$ is not Gromov $k$-hyperbolic (use that $k_{D_k}(z,z_j)\to\infty$ as $z_j\to\partial
D_k,$ $k=1,2$).\smallskip

As an immediate consequence we obtain that the polydisc is not Gromov $k$-hyperbolic. Moreover, even its ``symmetrized'' counterpart is not.

\begin{propozycja}\label{G_n}
$\mathbb{G}_{n}$ is not Gromov $c$- nor $k$-hyperbolic for $n\geq 2$.
\end{propozycja}

For the convenience of the reader, recall that the {\it symmetrized polydisc} $\mathbb{G}_{n},$ which is of great
relevance due to its properties and role (cf. \cite{AY}, \cite{Costara}), is the image of the holomorphic map
$$
\pi :\mathbb{D}^{n}\rightarrow\mathbb{C}^{n},\ \pi=(\pi_1,\ldots,\pi_n),
$$
$$\pi_k(z_1,\ldots,z_n)=\sum_{1\leq j_1<\ldots<j_k\leq n}z_{j_1}\ldots z_{j_k},\ z_1,\ldots,z_n\in\mathbb{D},\ 1\leq k\leq n,
$$
which is proper from $\mathbb{D}^{n}$ to $\mathbb{G}_{n}$.

Another interesting domain, the {\it tetrablock} (cf. \cite{AWY}), fails to be Gromov $k$-hyperbolic, too. Let
$$\varphi:\mathcal{R}_{II}\rightarrow \mathbb{C}^{3},\  \varphi(z_{11},z_{22},z)=(z_{11},z_{22},z_{11}z_{22}-z^{2}),$$
where $\mathcal{R}_{II}$ denotes the classical  Cartan domain of the second type (in $\mathbb{C}^{3}$), i.e.
$$\mathcal{R}_{II} =\{\widetilde{z}\in\mathcal{M}_{2\times 2}(\mathbb{C}): \widetilde{z}=\widetilde{z}^{t},\ \lVert \widetilde{z}\rVert<1\},$$
where $\lVert \cdot \rVert$ is the operator norm and $\mathcal{M}_{2\times 2}(\mathbb{C})$ denotes the space of $2 \times 2$ complex matrices (we identify a point $(z_{11},z_{22},z)\in\mathbb{C}^{3}$ with a $2\times 2$ symmetric matrix
$ \left( \begin{array}{ll}
                 z_{11}   &  z \\
                 z   &  z_{22}
\end{array} \right) ). $
Then  $\varphi$ is a proper holomorphic map and $\varphi(\mathcal{R}_{II})=\mathbb{E}$ is a domain, called the tetrablock.
\begin{propozycja}\label{tetrablock}
$\mathbb{E}$ is not Gromov $k$-hyperbolic.
\end{propozycja}

Since $\Bbb G_2$ and $\mathbb{E}$ are bounded $\mathbb{C}$-convex domains (see \cite[Theorem 1 (i)]{G2} and \cite[Corollary 4.2]{Zwonek}),
it follows that they are not Gromov $c^i$- nor $b$-hyperbolic either.

Buckley in \cite{Buckley}, claimed that it is because of the flatness of the boundary rather than the lack of smoothness that Gromov hyperbolicity fails. Recently, Gaussier and Seshadri have provided a proof of that conjecture. More precisely, their main result in \cite[Theorem 1.1]{Gaussier} states that any bounded convex domain in $\mathbb{C}^{n}$ whose boundary is $\mathcal{C}^{\infty}$-smooth and contains an analytic disc, is not Gromov $k$-hyperbolic. Lemma 5.4 in their proof used the $\mathcal{C}^{\infty}$ assumption in an essential way.
Our aim is to prove this result in a shorter way in $\mathbb{C}^{2}$, assuming only $\mathcal{C}^{1,1}$-smoothness. Moreover, the proofs of the facts we use are more elementary.
\begin{Theorem}\label{Gaussier}
Let $D$ be a convex domain in $\mathbb{C}^{2}$ containing no complex lines.\footnote{Then D is biholomorphic to a bounded domain
(cf. \cite[Theorem 7.1.8]{Jarnicki}).} Assume that $\partial D$ is $\mathcal{C}^{1,1}$-smooth and contains an analytic disc. Then $D$ is not Gromov $k$-hyperbolic.
\end{Theorem}

Besides, we give a partial answer to the question raised in \cite{Bonk}.
\begin{Theorem}\label{Pascal}
Let $D$ be a $\mathcal{C}^{1,1}$-smooth convex bounded domain in $\mathbb{C}^{2}$
admitting a defining function of the form $\varrho (z)=-\Re z_{1}+\psi (|z_{2}|)$
near the origin, where $\psi$ is a $\mathcal{C}^{1,1}$-smooth nonnegative convex function near $0$ satisfying $\psi (0)=0,$ and
\begin{equation}\label{infinite type}
\limsup _{x\rightarrow 0}\frac{\log \psi (|x|)}{\log |x|}=\infty.\footnote{If $\psi$ is $\mathcal{C}^{\infty},$ then $0$ is of infinite type if and only if condition (\ref{infinite type}) holds.}
\end{equation}
Then $D$ is not Gromov $k$-hyperbolic.
\end{Theorem}

Finally, note that there is no connection between Gromov hyperbolicity and pseudoconvexity. Indeed, take any strictly pseudoconvex domain $G.$ As we have already mentioned, $G$ is Gromov $k$-hyperbolic, and $k_G$ and $c_G$ are bilipschitz equivalent. Hence $G$ is Gromov $c$-hyperbolic, too.  Assume that,
respectively, $A\Subset G$ and $B$ is a relatively closed subset of $G$ such that $G\setminus A$ is a domain and that $B$ is negligible with respect to the $(2n - 2)$-dimensional Hausdorff measure.
Then $G\setminus A$ is Gromov
$c$-hyperbolic and $G\setminus B$ is Gromov $k$-hyperbolic, since
$$c_{G\setminus A} = c_{G}|_{(G\setminus A)\times (G\setminus A)}$$
(by the Hartogs extension theorem) and
$$
k_{G\setminus B} = k_{G}|_{(G\setminus B)\times (G\setminus B)}
$$
(cf. \cite[Theorem 3.4.2]{Jarnicki}).

However, the example with $G\setminus B$ does not have a smooth boundary.  The next proposition yields, in particular, a
family of non-pseudoconvex domains with smooth boundaries which are Gromov $k$-hyperbolic.

\begin{propozycja}\label{ball}
Let $G$ be a bounded domain in $\Bbb C^n\ (n\geq 2).$ Assume that $D\Subset G$ is a $\mathcal{C}^{2}$-smooth domain in
$\Bbb C^n$ and its Levi form has at least one positive eigenvalue at each boundary point. Then
$G\setminus D$ is a domain such that $k_{G\setminus D}$ is quasi-isometrically equivalent
to $k_{G}|_{(G\setminus D)\times (G\setminus D)}.$\footnote{One can show that these distances are not bilipschitz equivalent.}

In particular, if $G$ is Gromov $k$-hyperbolic, then so is $G\setminus\overline{D}.$
\end{propozycja}

\begin{corollary}\label{str_psc}
If $D\Subset G$ are strictly pseudoconvex domains in $\Bbb C^n,$ then $G\setminus\overline{D}$ is a Gromov $k$-hyperbolic domain.
\end{corollary}

The estimates that we use in the proof of Proposition 5 do not hold for the planar annulus $\mathbb{A}_r=\{z\in\mathbb{C}: r^{-1}<|z|<r\}$ ($r>1$).
However, any finitely connected proper planar domain is Gromov $k$-hyperbolic (cf. \cite[Proposition 3.2]{RT}).

\begin{propozycja}\label{compact}
Let $G$ be a bounded domain in $\Bbb C^n\ (n\geq 2).$ Assume that $K$ is compact subset of $G$ such that
through any point $z\in\Bbb C^n\setminus K$ passes a complex line disjoint from $K.$ Then
$G\setminus D$ is a domain such that $k_{G\setminus D}$ is quasi-isometrically equivalent to $k_{G}|_{(G\setminus D)\times (G\setminus D)}$.\footnote{One
can show that these distances are not bilipschitz equivalent if, for example, $K$ is a closed polydisc.}

In particular, if $G$ is Gromov $k$-hyperbolic, then so is $G\setminus K.$
\end{propozycja}

Note that we may take $K$ to be any compact ($\C$-)convex set, since
any compact or open $\C$-convex set $E$ in $\Bbb C^n$ is \emph{linearly convex}, i.e.
through any point in $\Bbb C^n\setminus E$ passes a complex line disjoint from $E$ (cf. \cite[Theorem 2.3.9]{APS}).

Throughout the paper $d_{D}$ denotes the (Euclidean) distance to $\partial D.$ A point $z\in\mathbb{C}^{n}$ we write as $(z_{1},\ldots,z_{n}),\ z_{j}\in\mathbb{C}.$

An appendix at the end of the paper includes some of the estimates for the Kobayashi distance and metric used in the proofs.

\section{Proofs}

\noindent{\it Proof of Proposition \ref{I}.}
Assume that $(X,d)$ is $\frac{\delta}{2}$-hyperbolic. Put $k=3+\delta$.
Then there are points $y_1,y_2\in X_2$ such that $d_2(y_1,y_2)=2s\geq 2k$.
Choose points $x_1,x_2^\ast\in X_1$ with $d_1(x_1,x_2^\ast)\geq 2s$.
By the path property of $X_1$, there is a $d_1$-continuous curve
$\gamma:[0,1]\to X_1$ joining the points $x_1$ and $x_2^\ast$
such that
$L_{d_1}(\gamma)<d_1(x_1,x_2^\ast)+1$. Note that $t\to
d_1(x_1,\gamma(t))$ is continuous.
Hence there is a smallest number $t_0$ such that $d_1(x_1,\gamma(t_0))=2s$.
Set $x_2=\gamma(t_0)$.

Now $L(\gamma|_{[0,t_0]}) \ge d_1(x_1,x_2)=2s$, and
$$
L(\gamma|_{[0,t_0]}) = L(\gamma) - L(\gamma|_{[t_0,1]})
\le d_1(x_1,x_2^\ast)+1 - d_1(x_2,x_2^\ast) \le d_1(x_1,x_2) +1.
$$
Let $t_1$ be the smallest number in $[0,t_0]$ such that
$d_1(x_1,\gamma(t_1))=s$. Set $x_3=\gamma(t_1)$. Then
$$
d_1(x_2,x_3) \ge d_1(x_1,x_2) - d_1(x_1,x_3) = s, \mbox{ and }
$$
$$
d_1(x_2,x_3) = L(\gamma|_{[0,t_1]}) = L(\gamma|_{[0,t_0]}) - L(\gamma|_{[t_1,t_0]})
\le 2s+1 - d_1(x_1,x_2) = s+1.
$$
Hence, $s= d_1(x_1,x_3)\leq
d_1(x_3,x_2)<s+1$.

Now define the following points in $X_1 \times X_2$:
$x=(x_1,y_1)$, $y=(x_2,y_1)$, $w=(x_3,y_1)$, and
$z=(x_3,y_2)$.
Then $d(z,w)=d(z,x)=d(z,y)=2s$ and $(x,y)_w\leq 1$,
$(x,z)_w= d(x,w)=s$, $(y,z)_w=d(y,w)\ge s-1$.
By the assumption of
$\frac{\delta}{2}$-hyperbolicity we reach the inequalities
$$
1\geq (x,y)_w\geq \min\{(y,z)_w,(x,z)_w\}-\delta\geq
s-1-\delta\ge 2
$$
which is a contradiction.

\begin{uwaga}
An essential ingredient in the proof of Proposition \ref{I} is the existence of points $x_{1},x_{2},x_{3}$ 
such the triangle inequality is a near-equality, namely
$(x_1, x_2)_{x_3} \le 1$.  The condition \eqref{wmp} is equivalent to
$(x_1, x_2)_{x_3} = o(d(x_1,x_2))$, and
$\left| d(x_1,x_3) - d(x_2,x_3)\right| = o(d(x_1,x_2))$.

Using this weaker hypothesis and following the steps of the above proof,
setting $2s =d(x_1,x_2)$ as before, we find
$$
o(s)\geq (x,y)_w\geq s - o(s) -\delta,
$$
leading to a contradiction when $s\to\infty$. Similar changes can be made in the proof below.
\end{uwaga}

\noindent{\it Proof of Proposition \ref{II}.}
Let first $(X_1,d_1)$ be $2\delta$-hyperbolic and $d_2\le 2c.$ Since $d\le d_1+2c,$ it follows that
$$(x_1,y_1)_{w_1}-2c\le (x,y)_w\le(x_1,y_1)_{w_1}+4c$$
and then $(X,d)$ is $(\delta+3c)$-hyperbolic.

Assume now that $(X,d)$ is $\delta$-hyperbolic. Following the proof of Pro\-position \ref{I}, we deduce that
one of the distances is bounded, say $d_2\le 2c.$ Then we get as above that $(X_1,d_1)$ is $(\delta+3c)$-hyperbolic.
\medskip

\noindent{\it Proof of Proposition \ref{G_n}.}
Let $a\in\Bbb D,$ $p_a=\pi (a,\ldots,a),$ $q_a=\pi (a,\ldots,a,-a)$ and $m_a=\pi (a,\ldots,a,0).$ We shall show that
$$S_{c_{\Bbb G_n}}(p_a,q_a,m_a,0)\to\infty\mbox{ as }|a|\to 1.$$
It follows exactly in the same way that $S_{k_{\Bbb G_n}}(p_a,q_a,m_a,0)\to\infty\mbox{ as }|a|\to 1.$
So, $\Bbb G_n$ is not Gromov $c$- nor $k$-hyperbolic for $n\geq 2.$

The holomorphic contractibility implies that
$$c_{\Bbb G_n}(p_a,q_a) \ge c_{\Bbb D}(a^n,-a^n)=2c_{\Bbb D}(a^n,0),$$
$$
\max\{c_{\Bbb G_n}(p_a,0), c_{\Bbb G_n}(q_a,0),c_{\Bbb G_n}(p_a,m_a),c_{\Bbb G_n}(q_a,m_a)\}\le c_{\Bbb D}(a^n,0).$$
Thus
$$S_{c_{\Bbb G_n}}(p_a,q_a,m_a,0)\ge c_{\Bbb G_n}(m_a,0)+2c_{\Bbb D}(a^n,0)-2c_{\Bbb D}(a,0).$$
Since
$$2c_{\Bbb D}(a,0)-2c_{\Bbb D}(a,0)\to\log n\mbox{ as }|a|\to 1,$$
it remains to see that $c_{\Bbb G_n}(m_a,0)\to\infty$ as $|a|\to 1.$ This follows by the fact that any
point $b\in\Bbb G_n$ is a weak peak point, i.e. there exists $f_b\in\mathcal O(\Bbb G_n,\Bbb D)$ such that
$|f_b(z)|\to 1$ as $z\to b$ (a consequence of \cite[Corollary 3.2]{Costara}).
\medskip

\noindent{\it Proof of Proposition \ref{tetrablock}.}
Let $a\in (0,1),$ and put $P_a=\varphi(\textup{diag}(a,a)),\ Q_a=\varphi(\textup{diag}(a,-a)).$
Recall that
$\Phi_a(Z)=(Z-a\textup{I})(\textup{I}-aZ)^{-1}$ is an automorphism of $\mathcal{R}_{II}.$ Direct computations show that
$$\varphi\circ\Phi_a  (\left( \begin{array}{ll}
                 z_{11}   &  z \\
                 z   &  z_{22}
\end{array} \right))=\varphi\circ\Phi_a  (\left( \begin{array}{ll}
                 z_{11}   &  -z \\
                 -z   &  z_{22}
\end{array} \right)),$$
whenever $\left( \begin{array}{ll}
                 z_{11}   &  z \\
                 z   &  z_{22}
\end{array} \right)\in\mathcal{R}_{II}.$ Thus, $\Phi_a$ induces an automorphism $\widetilde{\Phi}_{a}$ of $\mathbb{E}.$
It follows from this and \cite[Corollary 3.7]{AWY} that
\begin{multline*}
k_{\mathbb{E}}(0,(a,b,p))=\tanh^{-1}\max\Big{\{}\frac{|a-\overline{b}p|+|ab-p|}{1-|b|^{2}},\frac{|b-\overline{a}p|+|ab-p|}{1-|a|^{2}}\Big{\},}\\
(a,b,p)\in\mathbb{E},
\end{multline*}
$$
2k_{\mathbb{E}}(P_a,0),\ 2k_{\mathbb{E}}(Q_a,0),\ k_{\mathbb{E}}(P_a,Q_a)
=-\log d_{\mathbb{D}}(a) +\textup{O}(1).
$$
Observe that if $f(\lambda)=(0,\lambda,0),$, then $g_a=\widetilde{\Phi}_{-a}\circ f$ is a complex geodesic for $k_{\Bbb E}$
with $P_a=g_a(0)$, $Q_a=g\left(-\frac{2a}{1+a^2}\right)$.
Note that the Kobayashi middle point $R_{a}$ of $g_{a}|_{[-\frac{2a}{1+a^2},0]}$ tends to the boundary; more precisely,
$$R_{a}=g_{a}(-a)\rightarrow \textup{diag}(1,0) \textup{ as }a\rightarrow 1.$$
Consequently, $S_{k_{\Bbb E}}(P_a,Q_a,R_a,0)$ is comparable with $k_{\mathbb{E}}(R_a,0).$
By Pro\-position A1(b) (see Appendix), $k_{\mathbb{E}}(R_a,0)\to\infty$ as $a\to 1,$
which finishes the proof.
\medskip

\noindent{\it Proof of Theorem \ref{Gaussier}.} Since $\partial D$ contains an analytic disc, it is well known that it contains an affine disc (cf. \cite[Proposition 7]{NPZ}). We assume that this disc has center $0$ and lies in $\{z_{1}=0\},$ and that $D\subset\{\Re z_{1}>0\}.$

\begin{lemma}
We can find an $r>0$ such that for any $\delta >0$ small enough
there exist two discs
$\Delta (\tilde{p}_{\delta},r)$ and $\Delta (\tilde{q}_{\delta},r)$
in $D_{\delta}=D\cap\{z_{1}=\delta\}$ which touch $\partial{D}$ at two points $\hat{p}_{\delta}$ and $\hat{q}_{\delta}$ with $\lVert \hat{p}_{\delta}-\hat{q}_{\delta}\rVert >5r.$
\end{lemma}
\begin{proof}
We identify $\partial D\cap \{ z_{1}=0\}$ with a closed, bounded, convex subset of $\mathbb C$,
which is the closure of its interior. Call this interior $D_0.$

There exists $\zeta_0 \in D_0$ such that
$d_{D_0}(\zeta_0)=\max_{\zeta\in D_0}d_{D_0}(\zeta).$ Then
$$
M= \left\{ p \in \partial D_0 : |p- \zeta_0 | = \min_{\xi \in \partial D_0} |\xi- \zeta_0 |
\right\}
$$
is a not empty set which cannot be contained in any half plane
$$H_\theta = \{ \zeta :\\ \Re [(\zeta-\zeta_0)e^{-i\theta}] <0\}:$$ if it were,
one could find $\varepsilon>0$ such that
$d_{D_0}( \zeta_0 + \varepsilon e^{i\theta})
> d_{D_0}(\zeta_0)$. So there are $\hat p \neq \hat q
\in M$ such that $\arg( (\hat p -\zeta_0) (\hat q-\zeta_0)^{-1} ) \ge 2\pi/3$.
Take $r\in (0, \frac{\sqrt 3}{5+\sqrt 3}|\hat p -\zeta_0|)$,
$\tilde{p} = \zeta_0 + ( 1 - r|\hat p -\zeta_0|^{-1} )  (\hat p -\zeta_0)$,
and $\tilde{q}$ chosen likewise. Then the discs
$\Delta (\tilde{p},r)\subset D_{0}$ and $\Delta (\tilde{q},r)\subset D_{0}$
are tangent to $\partial D_{0}$ at $\hat{p}$ and $\hat{q}$.

Now we want to move these discs inside $D.$ By $\mathcal{C}^{1,1}$-smoothness of $D,$ we can move them (in $\mathbb{C}^{2}$) along
the vector $(1,0)$ inside $D,$ i.e. $\Delta (\tilde{p},r),\,\Delta (\tilde{q},r)\subset D\cap\{z_{1}=\delta\}=D_{\delta},$ for $0<\delta <\delta_{0}.$ If they do not touch $\partial D_{\delta},$ then shift them (separately at every sublevel set) to the boundary but leaving their centers on the real line passing through $\tilde{p}+(\delta,0)$ and $\tilde{q}+(\delta,0).$ Denote new discs by $\Delta (\tilde{p}_{\delta},r),\,\Delta (\tilde{q}_{\delta},r),$ and by $\hat{p}_{\delta},\,\hat{q}_{\delta}$ points of contact of those discs with $\partial D_{\delta}.$
\end{proof}

Choose now a point $a=(\delta_{0},0)\in D\textup{ (}\delta_{0}>0\textup{)}$ and consider the cone with vertex at $a$ and base $\partial D\cap\{z_{1}=0\}.$ Denote by $G_{\delta}$ the intersection of this cone and $\{z_{1}=\delta\}.$ For any $\delta>0$ small enough the line segment with ends at $\tilde{p}_{\delta}$ and $\hat{p}_{\delta}$ intersects $\partial G_{\delta},$ say at $p_{\delta}.$ Define $q_{\delta}$ in a similar way.

Set $\tilde{s}_{\delta}=\frac{ \tilde{p}_{\delta}+\tilde{q}_{\delta}}{2}.$ We shall show that $S_{k_D}(p_{\delta},q_{\delta}, \tilde{s}_{\delta}, a) \rightarrow \infty$ as $\delta\rightarrow 0.$
For this we will see that $(p_{\delta},\tilde{s}_{\delta})_{a}-(p_{\delta},q_{\delta})_{a}\rightarrow \infty$ as $\delta\rightarrow 0.$ It will follow in the same way that $(q_{\delta},\tilde{s}_{\delta})_{a}-(p_{\delta},q_{\delta})_{a}\rightarrow \infty.$

It is enough to prove that
\begin{equation}\label{1}
k_{D}(q_{\delta},a)-k_{D}(\tilde{s}_{\delta},a)<c_{1}
\end{equation}
and
\begin{equation}\label{2}
k_{D}(p_{\delta},q_{\delta})-k_{D}(p_{\delta},\tilde{s}_{\delta})\rightarrow \infty.
\end{equation}
Here and below $c_{1},c_{2},\ldots$ denote some positive constants  which are independent of $\delta.$

For (\ref{1}), observe that, by  Propositions A1(a) and A2 (see Appendix),
\begin{equation}\label{Kobayashi estimates}
k_{D}(\tilde{s}_{\delta},a)\geq\frac{1}{2}\log\frac{d_{D}(a)}{d_{D}(\tilde{s}_{\delta})}\textup{ and }2k_{D}(q_{\delta},a)\leq -\log d_{D}(q_{\delta})+c_{2}.
\end{equation}
It remains to use that $d_{D}(\tilde{s}_{\delta})=d_{D}(q_{\delta})$ for any $\delta >0$ small enough.

To prove (\ref{2}), denote by $F_{\delta}$ the convex hull of
$\Delta (\tilde{p}_{\delta},r)$ and $\Delta (\tilde{s}_{\delta},r).$ Then
by inclusion
$k_{D}(p_{\delta},\tilde{s}_{\delta})\leq k_{F_{\delta}}(p_{\delta},\tilde{s}_{\delta})$.

\begin{lemma}
$k_{F_{\delta}}(p_{\delta},\tilde{s}_{\delta}) < -\frac{1}{2}\log d^{\prime}_{D}(p_{\delta})+c_{3},$
where $d^{\prime}_{D}$ is the distance to $\partial D$ in the $z_{2}$-direction.
\end{lemma}

{\it Proof.} For $\delta>0$ small enough we have that
$$
d^{\prime}_{D}(p_{\delta}) = d_{D_\delta}(p_{\delta}) =
d_{F_\delta}(p_{\delta}) = d_{\Delta (\tilde{p}_{\delta},r)}(p_{\delta})
$$
because the closest point on $\partial D_\delta$ belongs to $\partial \Delta (\tilde{p}_{\delta},r)$.
Now $k_{F_{\delta}}(p_{\delta},\tilde{s}_{\delta}) \le
k_{F_{\delta}}(p_{\delta},\tilde{p}_{\delta}) + k_{F_{\delta}}(\tilde p_{\delta},\tilde{s}_{\delta}) $.

Since $\Delta (\tilde{p}_{\delta},r)\subset F_{\delta}$,
\begin{multline*}
k_{F_{\delta}}(p_{\delta},\tilde{p}_{\delta})
\le k_{\Delta (\tilde{p}_{\delta},r)}(p_{\delta},\tilde{p}_{\delta})
= \frac{1}{2}\log \frac{1+\frac{|p_{\delta}-\tilde{p}_{\delta}|}{r} }{1-\frac{|p_{\delta}-\tilde{p}_{\delta}|}{r}}
\\
\le - \frac{1}{2} \log d_{\Delta (\tilde{p}_{\delta},r)}(p_{\delta}) +\frac12 \log (2r)
= -\frac{1}{2}\log d^{\prime}_{D}(p_{\delta})+ \frac12 \log (2r) .
\end{multline*}
On the other hand, by using a finite chain of discs of radius $r$ with
centers on the line segment from $\tilde{p}_{\delta}$ to $\tilde{s}_{\delta}$, we obtain that
$$k_{F_{\delta}}(\tilde p_{\delta},\tilde{s}_{\delta}) \le
4 \frac{|\tilde p_{\delta}-\tilde{s}_{\delta}|}{r} \le C(r).\qed$$
\smallskip

Now, we shall show that
\begin{equation}\label{3}
2k_{D}(p_{\delta},q_{\delta})>-\log d^{\prime}_{D}(p_{\delta})-\log d^{\prime}_{D}(q_{\delta})-c_{4},
\end{equation}
which implies (\ref{2}), because $d^{\prime}_{D}(q_{\delta})\rightarrow 0$ as $\delta\rightarrow 0.$

Since the Kobayashi distance is intrinsic,
we may find a point $m_{\delta}\in D$ such that
$$\lVert p_{\delta} - m_{\delta}\rVert=\lVert q_{\delta}-m_{\delta}\rVert\geq\frac{\lVert p_{\delta}-q_{\delta}\rVert}{2},$$
$$
k_{D}(p_{\delta},q_{\delta})>
k_{D}(p_{\delta},m_{\delta})+k_{D}(m_{\delta},q_{\delta})-1.
$$

Let $\check{p}_{\delta}\in\partial D$ be the closest point
to $p_{\delta}$ in the direction of the complex line through $p_{\delta}$ and $m_{\delta}.$

Recall that $d_D'$ is the distance to $\partial D$ in the $z_{2}$-direction and $d_D(p_\delta)$ is attained in $z_{1}$-direction
for any $\delta>0$ small enough.
This means that the standard basis is adapted to the local geometry of $\partial D$ near $p_\delta$,
and more precisely, if $X=(X_1,X_2)\in\Bbb C^2$ is a unit vector,
\cite[(4)]{NPZ} states in this case that there exists a constant $C$ such that
\begin{equation*}
\frac{1}{d_D(p_\delta,X)}\le \frac{|X_1|}{d_{D}(p_\delta)} +
\frac{|X_2|}{d'_{D}(p_\delta)}\le\frac{C}{d_D(p_\delta,X)},
\end{equation*}
where $d_D(\cdot;X)$ is the distance to $\partial D$ in direction $X.$
Since $d'_{D}\ge d_D$, we obtain
$$d_D(p_\delta;X)\le c_5d_D'(p_\delta).$$

Let $X= \frac{m_\delta-p_{\delta}}{\|m_\delta-p_{\delta}\|}.$ Then
$\| p_{\delta}-\check{p}_{\delta}\| = d_X (p_{\delta})$ and thus
\begin{equation}\label{4}
\lVert p_{\delta}-\check{p}_{\delta}\rVert< c_{5}d^{\prime}_{D}(p_{\delta}).
\end{equation}

By convexity, $D$ is on the one of the sides, say $H_{\delta},$ of the real tangent plane to $\partial D$ at $\check{p} _{\delta}.$
Since $\frac{\lVert m_{\delta}-\check{p}_{\delta}\rVert}{d_{H_\delta}(m_{\delta})}
=\frac{\lVert p_{\delta}-\check{p}_{\delta}\rVert}{d_{H_\delta}(p_{\delta})}$, it follows by
\eqref{Kobayashi estimates} that
\begin{equation}
2k_{D}(p_{\delta},m_{\delta})
\geq 2k_{H_{\delta}}(p_{\delta},m_{\delta})
\geq \log \frac{d_{H_\delta}(m_{\delta})}{d_{H_\delta}(p_{\delta})}
=
\log\frac{\lVert m_{\delta}-\check{p}_{\delta}\rVert}{\lVert p_{\delta}-\check{p}_{\delta}\rVert}.
\end{equation}
Applying the triangle inequality and \eqref{4}, we get that
\begin{multline*}
\log\frac{\lVert m_{\delta}-\check{p}_{\delta}\rVert}{\lVert p_{\delta}-\check{p}_{\delta}\rVert}\ge
\log\frac{\lVert m_\delta -p_\delta\rVert -\lVert p_\delta -\check{p}_\delta\rVert}{\lVert p_{\delta}-\check{p}_{\delta}\rVert}\ge\\
\log\left(\frac{r}{2\lVert p_{\delta}-\check{p}_{\delta}\rVert}-1\right)\ge\log\frac{r}{2c_5 d^{\prime}(p_\delta)}\,-\,1,
\end{multline*}
for any $\delta >0$ small enough.
So $2k_{D}(p_{\delta},m_{\delta})>-\log d^{\prime}_{D}(p_{\delta})-c_{6}.$ Similarly, $2k_{D}(q_{\delta},m_{\delta})>-\log d^{\prime}_{D}(q_{\delta})-c_{6},$ which implies (\ref{3}), and completes the proof.

\begin{uwaga}
All the above arguments hold in $\mathbb{C}^{n}$, $n\ge 3$, except (\ref{4}).
\end{uwaga}

\noindent{\it Proof of Theorem \ref{Pascal}.} Since the case when $\psi(z_{0}) =0$ for some $z_{0}\not= 0,$ is covered by Proposition \ref{Gaussier}, we may assume that $\psi^{-1}\{0\}=\{0\}$. Also assume $p=(1,0)\in D.$

Let $\alpha (x),$ small enough, an increasing function such that for any $x>0,\ \psi^{\prime}(x)\geq\psi^{\prime}((1-\alpha (x))x)\geq\frac{1}{2}\psi' (x)$.
We choose, for $x>0,\ q(x)=(\psi(x),0),\ r(x)=(\psi (x),-(1-\alpha (x))x),\ s(x)=(\psi (x),(1-\alpha (x))x).$

We claim that for $x$ small enough:
\begin{enumerate}
\item $d_{D}(q)=\psi (x),$\label{A}
\item $\frac{\alpha (x)}{4}x\psi^{\prime}(x)\leq d_{D}(s),d_{D}(r)\leq \alpha (x)x\psi^{\prime}(x),$\label{B}
\item the functions $k_{D}(s,q)+\frac{1}{2}\log\alpha (x)$ and $k_{D}(r,q)+\frac{1}{2}\log\alpha (x)$ are bounded,\label{C}
\item the function $k_{D}(r,s)+\log\alpha (x)$ is bounded.\label{D}
\end{enumerate}

Before we proceed to prove the claims we make some general observation about infinite order of vanishing.
\begin{lemma}
\label{psipsi}
For any  $\varepsilon>0$ and $A>0$, there exists $x \in (0,\varepsilon)$
such that $\frac{x \psi'(x) }{\psi(x)} >A$.
\end{lemma}
\begin{proof}
Suppose instead that there exist $\varepsilon>0$ and $A>0$ such that
$ \frac{x \psi'(x) }{\psi(x)} \le A$ for $0<x\le \varepsilon$.  Then
$$
\frac{d}{dx} \left( \log \psi(x)\right) \le \frac{A}x, \quad 0<x\le \varepsilon,
$$
so $\log(\psi(\varepsilon)) - \log (\psi(x)) \le A \left(  \log \varepsilon - \log x \right)$,
i.e.
$$
 \psi(x) \ge \frac{\psi(\varepsilon)}{\varepsilon^{A}} x^{ A}, \quad 0<x\le \varepsilon,
$$
which means that at the point $0$ there is finite order of contact
with the tangent hyperplane, a contradiction.
\end{proof}

Assume the claims for a while,
 and observe that for any $x$ verifying the conclusion of Lemma \ref{psipsi} we have
$$
(r,p)_{q}-(r,s)_{q}, \quad (p,s)_{q}-(r,s)_{q}\geq -\frac{1}{2}\log\frac{\psi (x)}{x\psi^{\prime}(x)}+C_{1}.
$$
Since the above quantity can be made arbitrarily large, it finishes the proof.

It remains to prove (\ref{A})-(\ref{D}).

(\ref{A}) is clear. Next, since $(\psi((1-\alpha (x))x),(1-\alpha (x))x)\in D,\ d_{D}(s)\leq \psi (x)-\psi ((1-\alpha (x))x))\leq \alpha (x)x\psi^{\prime}(x)$ by convexity. Let $L$ be the real line through $(\psi((1-\alpha (x))x),(1-\alpha (x))x)$ and $(\psi (x),x).$ Its slope is less than $\psi^{\prime}(x),$ so $d_{D}(s)\geq \textup{dist}\,(s,L^{\prime}),$ where $L^{\prime}$ is the line through $(\psi((1-\alpha (x))x),(1-\alpha (x))x)$ with slope $\psi^{\prime}(x),$ so
\begin{multline*}
d_{D}(s) \geq \frac{\psi (x) - \psi ((1-\alpha(x))x)}{\sqrt{1+\psi^{\prime}(x)^{2}}} \\
\geq \frac{1}{2} {\alpha (x) \times \psi^{\prime}((1-\alpha (x))x)} \geq \frac{1}{4} \alpha (x) \times \psi^{\prime}(x).
\end{multline*}
Thus, $\frac{\alpha (x)}{4}x\psi^{\prime}(x)\leq d_{D}(s)\leq \alpha (x)x\psi^{\prime}(x).$ Analogous estimates hold for $r,$ which gives (\ref{B}).

The analytic disc $\zeta \mapsto (\psi (x),x\zeta)$ provides immediate upper estimates in (\ref{C}) and (\ref{D}).

To get lower estimate for $k_{D}(s,q),$ we map $D$ to a domain in $\mathbb{C}$ by the complex affine projection $\pi_{s}$ to $\{z_{1}=\psi (x)\},$ parallel to the complex tangent space to $\partial D$ at the point $(\psi(x),x).$ Then $\pi_{s}(D)=\{\psi (x)\}\times D_{s},$ where $D_{s}$ is a convex domain in $\mathbb{C},$ containing the disc $\{|z_{2}|<x\},$ and its tangent line at the point $x$ is the real line $\{\Re z_{2}=x\}.$ The projection is given by the explicit formula
$$\pi_{s}(z_{1},z_{2})=\Big{(}\psi (x),z_{2}+\frac{\psi (x)-z_{1}}{\psi^{\prime}(x)}\Big{)}.$$
We renormalize by setting $f_{+}(z)=1-\frac{1}{x}[\pi_{s}(z)]_{2}.$
Therefore $f_{+}(D) \subset H=\{z\in\mathbb{C}:\Re z>0\},$ so
\begin{equation}\label{E}
k_{D}(s,q)\geq
 k_{H}(f_{+}(s),f_{+}(q)) = k_H (\alpha(x),0)
 \geq -\frac{1}{2}\log \alpha (x) +C_{2},
\end{equation}
where $C_{2}>0$ does not depend on $x.$

The estimate for $k_{D}(r,q)$ proceeds along the same lines, but we use the projection $\pi_{r}$ to $\{z_{1}=\psi (x)\}$ along the complex tangent space to $\partial D$ at $(\psi (x),x),$ given by
$$\pi_{r}=\Big{(}\psi (x),z_{2}-\frac{\psi (x)-z_{1}}{\psi^{\prime}(x)}\Big{)}.$$
Note that choosing $f_{-}(z)=1+\frac{1}{x}[\pi_{r}(z)]_{2},$ we have $f_{-}(D)\subset\{\Re z>0\}.$

Now we tackle the lower estimates for $k_{D}(r,s).$ Let $\gamma$ be any piecewise $\mathcal{C}^{1}$ curve such that $\gamma (0)=s,\,\gamma (1)=r.$
Let $c_{0}<\frac{1}{2}.$
We claim that there exists $t_{0}\in (0,1)$ such that if we set $u=\gamma (t_{0}),$ then $|f_{+}(u)|,|f_{-}(u)|\geq c_{0}.$

For this write $\gamma=(\gamma_{1},\gamma_{2}).$ Set $\zeta_{1}=1-\frac{\psi (x)\,-\,\gamma_1(t_0)}{x\psi^{\prime}(x)}.$
By the explicit form of $\pi_{s},$ the condition $|f_{+}(u)|\geq c_{0}$ reads
$|\zeta_{1}-\frac{\gamma_{2}(t_0)}{x}|\geq c_{0},$
and the condition $|f_{-}(u)|\geq c_{0}$ reads $|\zeta_{1}+\frac{\gamma_{2}(t_0)}{x}|\geq c_{0}.$
We claim that the discs $\overline{\mathbb{D}}(\zeta_{1},c_{0})$
and $\overline{\mathbb{D}}(-\zeta_{1},c_{0})$ are disjoint for any $t.$
Indeed, they would intersect if and only if
$0\in\overline{\mathbb{D}}(\zeta_{1},c_{0}),$ which implies
$$
\Re\Big{(}\frac{\gamma_{1}(t_0)}{x \psi^{\prime}(x)}\Big{)}
\leq -1+c_{0}+\frac{\psi (x)}{x\psi^{\prime}(x)}\leq -\frac{1}{3}
$$
for any $x$ such that $\frac{\psi (x)}{x\psi^{\prime}(x)}\leq \frac{1}{6}$,
which we may assume
by Lemma \ref{psipsi}.
In particular $\Re\gamma_{1}(t_0)<0,$ which is excluded for any $\gamma (t)\in D.$
Now let $t_{1}=\max\{t:\frac{\gamma_{2}(t)}{x}\in\overline{\mathbb{D}}(\zeta_{1},c_{0})\}.$
Then $\frac{\gamma_{2}(t_{1})}{x}\notin\overline{\mathbb{D}}(-\zeta_{1},c_{0}),$
and by continuity there is $\eta >0$ such that $\frac{\gamma_{2}(t_{1}+\eta)}{x}\notin\overline{\mathbb{D}}(-\zeta_{1},c_{0}),$ and of course $\frac{\gamma_{2}(t_{1}+\eta)}{x}\notin\overline{\mathbb{D}}(\zeta_{1},c_{0})$ by maximality of $t_{1},$ so $t_{0}=t_{1}+\eta$ will
provide a point satisfying the claim.

Consequently,
taking a curve $\gamma$ such that
$$k_D(r,s)+1 >  \int_{0}^{1}\kappa_{D}(\gamma (t);\gamma^{\prime}(t))dt,$$
\begin{multline*}
\int_{0}^{1}\kappa_{D}(\gamma (t);\gamma^{\prime}(t))dt
\geq\int_{0}^{t_{0}}\kappa_{D}(\gamma (t);\gamma^{\prime}(t))dt
+\int_{t_{0}}^{1}\kappa_{D}(\gamma (t);\gamma^{\prime}(t))dt  \\
\geq k_{D}(r,u)+k_{D}(u,s).
\end{multline*}
We end the proof by estimating $k_{D}(r,u)$ in the same way as we did $k_D(r,q)$
above, and $k_{D}(u,s)$ as as we did $k_D(s,q)$
above, using the maps $f_+,f_-$ and estimates about the Kobayashi distance 
in a half plane.
\medskip

\noindent{\it Proof of Proposition \ref{ball}.} Set $G'=G\setminus\overline{D}.$

Assume first that $G'$ is not a domain. Let $G''$ be a bounded connected component
of $G'.$ Consider a farthest point $a\in\partial G''$ from the origin. Then $a$ is a concave boundary point
of $D$ which a contradiction.

Choose now a smooth domain $E$ such that $D\Subset E\Subset G.$ By smoothness and compactness, there is a constant $C>0$
such that any two points in $G'\cap\overline{E}$ may be jointed by a path in $G'\cap\overline{E}$ of length (at most) $C.$
By Propositions A3 (after integration) and A4, we may find a constant
$c>0$ such that
$$k_{G'}(z,w)\le c||z-w||^{1/4} \le c^2,\ z,w\in G'\cap\overline{E},$$
$$k_{G'}\le c k_G,\ z,w\in G\setminus E.$$
So there are constants $c_1, c_2$ such that
$$k_G\le k_{G'}\le c_1 k_G + c_2,\ z,w\in G\setminus E\mbox{ or }z,w\in G'\cap\overline{E}.$$

Finally, let $z\in G\setminus E$ and $w\in G'\cap\overline{E}.$ Since the Kobayashi distance is intrinsic
and $\partial E$ is  compact, we may find a point $u\in\partial E$ such that
$$k_G(z,w)=k_G(z,u)+k_G(u,w).$$
It follows that
$$k_G(z,w)\ge ck_{G'}(z,u)+ck_{G'}(u,w)\ge ck_{G'}(z,w).$$

\noindent{\it Proof of Proposition \ref{compact}.} It clear that $G'=G\setminus K$ is a domain. Following the previous proof,
let $E$ be a domain such that $K\subset E\Subset G.$ All the arguments in the previous proof work except the estimate on
$G'\cap\overline{E}.$
We need to find a constant $c>0$ such that
$$k_{G'}(z,w)\le c,\ z,w\in G'\cap\overline{E}.$$

Take a complex line $L$ through $z$ which is disjoint
from $K.$ Then the disc in $L$ with center of $z$ and radius $d_{G\cap L}(z)$ lies in $G'.$
Choose a common point $z'$ of this disc and $\partial E$ such that $||z-z'||=d_{E\cap L}(z).$ Then
$$\tanh l_{G'}(z,z')\le\frac{||z-z'||}{d_{G\cap L}(z)}\le
 1-\frac{r}{d_{G\cap L}(z)}\le 1-\frac{2r}{s},$$
where $r=\mbox{dist}(E,\partial G)$ and $s=\mbox{diam }G.$

Choosing $w'$ for $w$ in the same way, it follows that
$$k_{G'}(z,w)\le k_{G'}(z,z')+k_{G'}(z',w')+k_{G'}(w',w)\le 2c'+c''$$
where $c'=\tanh^{-1}(1-2r/s)$ and $c''=\max k_{G'}|_{\partial E\times\partial E}.$

\section{Appendix}

\begin{proposition} \textup{(cf. \cite[Proposition 2]{NT})}\hfill

(a) Let $D$ be proper convex domain in $\Bbb C^n.$ Then
$$c_{D}(z,w)\geq\frac{1}{2}\left|\log\frac{d_{D}(z)}{d_{D}(w)}\right|,\quad z,w\in D.$$

(b) Let $D$ be proper $\Bbb C$-convex domain in $\Bbb C^n.$ Then
$$c_{D}(z,w)\geq\frac{1}{4}\left|\log\frac{d_{D}(z)}{d_{D}(w)}\right|,\quad z,w\in D.$$
\end{proposition}

\begin{proposition} \textup{(see the proof of \cite[Proposition 10.2.3]{Jarnicki})} Let $b$ be a $\mathcal{C}^{1,1}$-smooth boundary point
of a domain $D$ is $\Bbb C^n$ and let $K\Subset D.$ Then there exist a neighborhood $U$ of $b$ and a constant $C>0$ such that
$$2k_D(z,w)\le-\log d_D(z)+C,\quad z\in D\cap U,\ w\in K.$$
\end{proposition}

\begin{proposition} \textup{(cf. \cite[Theorem 1]{DNT})} Let $b$ is a $\mathcal{C}^2$-smooth non-pseu\-doconvex boundary point of
a domain $D$ in $\Bbb C^2.$ Then there exist a neighborhood $U$ of $b$ and a constant $c>0$ such that
$$
c\kappa_D(z;X)\le\frac{|\langle\nabla d_D(z),X\rangle|}{(d_D(z))^{3/4}}+|X|,\quad z\in D\cap U,\ X\in\Bbb C^n.
$$
\end{proposition}

\begin{proposition} Let $D$ be a bounded domain in $\Bbb C^n.$ Let $U$ and $V$ be neighborhoods of $\partial D$ with
$V\Subset U.$ Then there exists a constant $c>0$ such that for any connected component $D'$ of $D\cap U$ one has that
$$ck_{D'}(z,w)\le k_D(z,w),\quad z,w\in D'\cap V.$$
\end{proposition}

\noindent{\it Proof.} Let $\varepsilon>0.$ Take a smooth curve $\gamma:[0,1]\to D$ such that $\gamma(0)=z,$ $\gamma(1)=w$ and
$$k_D(z,w,\varepsilon):=k_D(z,w)+\varepsilon>\int_0^1\kappa_{D}(\gamma (t);\gamma^{\prime}(t))dt.$$
Let $s=\sup\{t\in(0,1):\gamma(0,t)\subset D'\cap V\}$ and $r=\inf\{t\ge s:\gamma([t,1])\subset D'\cap V\}.$
Set $z'=\gamma(s)$ and $w'=\gamma(r).$ The localization property of the Kobayashi metric (cf. \cite[Proposition 7.2.9]
{Jarnicki}) provides a constant $c'>0$ such that
$$c'\kappa_{D'}(u;X)\le\kappa_D(u;X),\quad z\in {D'}\cap V,\ X\in\Bbb C^n.$$
It follows that
\begin{multline*}
k_D(z,w,\varepsilon)>c'k_{D'}(z,z')+k_D(z',w')+c'k_{D'}(w',w)\\
\ge c'k_{D'}(z,w)+k_D(z',w')-c'k_{D'}(z',w').
\end{multline*}

If $z'\neq w',$ then $z',w'\in {D'}\cap\partial V\Subset {D'}.$
Then there exists a constant $c_1>0$ such that
$$k_{D'}(u,v)\le c_1||u-v||,\quad u,v\in {D'}\cap\partial V.$$ On the other hand, since $D$ is bounded,
we may find a constant $c_2>0$ such that
$$k_D(u,v)\ge c_2||u-v||,\quad u,v\in {D'}\cap\partial V.$$

Then $$k_D(z,w,\varepsilon)>c'k_{D'}(z,w)+(c_2-c'c_1)||z'-w'||.$$
Since $$k_D(z,w,\varepsilon)>k_D(z',w')\ge c_2||z'-w'||,$$ we
get that
$$k_D(z,w,\varepsilon)>c'k_{D'}(z,w)-(c'c_1/c_2-1)^+k_D(z,w,\varepsilon).$$

The last inequality also holds if $z'=w'.$ Letting $\varepsilon\to 0,$ we obtain that
$$k_D(z,w)\ge\min\{c',c_2/c_1\}k_{D'}(z,w).$$

\end{document}